\documentclass[english, 12pt]{article}
\usepackage{amsfonts}
\usepackage{amsmath}
\usepackage{amssymb}
\usepackage{amsthm}
\usepackage{amscd}
\usepackage{amscd}
\usepackage{amsthm}

\newtheorem{defi}{Definition}[section]
\newtheorem{theo}{Theorem}[section]
\newtheorem{prop}{Proposition}[section]
\newtheorem{lem}{Lemma}[section]

\newtheorem{rk}{Remark}[section]

\newtheorem{exa}{Example}[section]

\numberwithin{equation}{section}

\def\A{{\cal{A}}}

\def\R{{\mathbb{R}}}

\def\calE{{\cal{E}}}

\def\calL{{\cal{L}}}

\newcommand{\ui}{{u_{\infty}}}
\newcommand{\Wa}{W^{\alpha/2,2}}

\newcommand{\Wo}{W_0^{\alpha/2,2}(\Omega)}

\newcommand{\lam}{\lambda}

\newcommand{\Om}{\Omega}

\newcommand{\alp}{\alpha}

\begin{document}
\bibliographystyle{alpha}

\title{ The heat equation for the  Dirichlet fractional Laplacian with negative potentials:
Existence and blow-up of nonnegative solutions}

\author{\normalsize   Ali BenAmor\footnote{corresponding author} \footnote{Department of Mathematics, Faculty
of Sciences of Gab\`es. Uni.Gab\`es, Tunisia. E-mail: ali.benamor@ipeit.rnu.tn}
\& Tarek Kenzizi\footnote {Department of Mathematics, Faculty of Sciences of Tunis, Uni. Tunis ElManar, Tunisia }
}

\date{}
\maketitle
\begin{abstract} We establish conditions ensuring either existence or blow-up of nonnegative solutions
for the heat equation generated by the  Dirichlet fractional Laplacian perturbed by negative
potentials  on bounded sets. The elaborated theory is supplied by some examples.
\end{abstract}

\section{Introduction}
In this paper, we discuss the question of  existence as well as blow-up  of nonnegative
solutions for negatively perturbed Dirichlet fractional Laplacian on open bounded subsets of $\R^d$.\\
For every $0<\alpha<\min(2,d)$ and every open bounded subset $\Om\subset\R^d$, we designate
by  $L_0:=(-\Delta)^{\frac{\alpha}{2}}|_\Omega$ the fractional Laplacian with zero Dirichlet
condition on $\Omega^c$ (as explained in the next section). We consider the associated perturbed heat
(or parabolic) equation
\begin{eqnarray}
\label{heat1}
\left\{\begin{gathered}
-\frac{\partial u}{\partial t}=L_0u - Vu,
\quad \hbox{in } (0,T)\times\Omega,\\
u(t,\cdot)=0,\ on~~~\Omega^c,\ \forall\,0<t<T\leq\infty\\
u(0,\cdot)= u_{0},~~~{\rm } \in \Omega,
\end{gathered}
\right.
\end{eqnarray}
where $u_{0}\geq 0$ is a Borel measurable square integrable  function on $\Om$ and $V\in L^1_{loc}(\Om)$ is a positive
Borel function. The meaning of a solution for the equation (\ref{heat1}) will be explained in the next section.\\
We emphasize that the potential $V$ is not supposed to be in the (generalized) Kato class, so that the
standard perturbation theory of Dirichlet forms does not help any more to decide wether a nonnegative solution occurs or not.\\
Even for the special case of a  Hardy potential
\begin{eqnarray}
V(x)=\frac{c}{|x|^\alpha},\ x\neq 0,\ c>0 
\end{eqnarray}
only partial information about the problem are established in the literature, to our best knowledge. Indeed,
if $0\in\Omega$ one derives from the paper of Beldi--Belhaj Rhouma--BenAmor \cite{benamor-JPA}, that
for every $0<c\leq c^*:=\frac{2^\alpha\Gamma^2(\frac{d+\alpha}{4})}{ \Gamma^2(\frac{d-\alpha}{4})}$, a nonnegative
solution exists. However, there is no answer for $c>c^*$.\\
Our main task in this paper is to shed some light towards solving the problem by giving conditions ensuring existence
as well as blow up of nonnegative solutions for (\ref{heat1}).\\
Focusing on nonnegative solutions is motivated, among
other reasons, by the fact that they are physical solutions on one hand and on the other one by  the
significance of the above considered operators in physics and in other area of natural sciences. For instance, the
case $\alpha=1$ corresponds to the nonrelativistic Schr\"odinger equation, whereas the general case models the
so called L\'evy motion or L\'evy flights. For more about aspects of applications of the fractional Laplacian we
refer the reader to \cite{dubkov}. A prototype of application in biology is illustrated in \cite{nicolas}.\\
The inspiring point for us was the papers of Baras--Goldstein \cite{baras-goldstein}, Cabr\'e--Martel
\cite{cabre-martel} and  Goldstein--Zhang \cite{goldstein-zhang} where the problem was addressed and solved for
the Dirichlet Laplacian (i.e. $\alpha=2$) on Lipschitz domains. In the latter cited papers, the authors proved, in particular, that existence and nonexistence  of nonnegative solutions in the case where the principal part of the equation is
the Dirichlet Laplacian (or an elliptic operator) is related to the size of the bottom of the spectrum of the
operator $-\Delta|_\Omega-V$. However, there is a substantial difference between the Laplacian and the fractional Laplacian. Whereas it is known that the first one is local and therefore suitable for describing diffusions, the
second one is nonlocal and commonly used for describing superdiffusions (L\'evy flights). These differences are
reflected in the way of computing for both operators (Green formula, integration by part, Leibnitz formula....).\\
Nonetheless, we shall show that the method used in \cite{cabre-martel,goldstein-zhang} still apply in our
setting. Especially, for the instantaneous blow-up  part, which is the major novelty of this paper, adequate generalization of the intermediate results to the nonlocal case are established (see in particular, formula (\ref{en-in}) and Theorem (\ref{Log-Sol})). These extensions make it possible to carry over the method to the nonlocal setting.\\
As a conclusion we approve that the used method provides a unified approach for handling the local as well as
the nonlocal case.

\section{Preparing results}
To state our main results, it is convenient to introduce the following notations.\\
From now on we fix an open bounded subset $\Omega\subset\R^d$ with Lipschitz boundary, a real number $\alpha$ such that
$0<\alpha<\min(2,d)$.\\
The Lebesgue spaces $L^2(\R^d,dx)$, resp. $L^2(\Omega,dx)$ will be denoted by $L^2$, resp. $L^2(\Omega)$ and their respective norms will be denoted by $\|\cdot\|_{L^2}$, resp. $\|\cdot\|_{L^2(\Om)}$ .
We shall write $\int\cdots$ as a shorthand for $\int_{\R^d}\cdots$.\\
The letters  $C$, $c$, $C^{'}$ will denote generic positive constants
which may vary in value from line to line. We shall also use the notation $f\sim g$ to mean that there are  constants $c,c'$ such that $cg\leq f\leq c'g$.\\
As far as concepts related to Dirichlet forms we refer the reader to \cite{fukushima-book}.\\
Consider the bilinear symmetric form $\calE$ defined in $L^2$ by
\begin{eqnarray}
\calE(f,g)&=&\frac{1}{2}{\A} (d,\alpha)\int \int \frac{(f(x)-f(y))(g(x)-g(y))}
{|x-y|^{d+\alpha}}\,dxdy,\nonumber\\
D(\calE)&=&W^{\alpha/2,2}(\R^d)
:=\{f\in L^2\colon\,\calE[f]:=\calE(f,f)<\infty\},\,
\label{formula1}
\end{eqnarray}
where
\begin{eqnarray}
{\A}{(d,\alpha)}=\frac{\alpha\Gamma(\frac{d+\alpha}{2})}
{2^{1-\alpha}\pi^{d/2}\Gamma(1-\frac{\alpha}{2})}.
\label{analfa}
\end{eqnarray}
Using Fourier transform $\hat f(\xi)=(2\pi)^{-d/2}\int e^{-ix\cdot\xi}f(x)\,dx$, a straightforward computation yields the following identity
(see \cite[Lemma 3.1]{frank})
\begin{eqnarray}
\int |\xi|^\alpha|\hat f(\xi)|^2\,d\xi=\calE[f],\ \forall\,f\in W^{\alpha/2,2}(\R^d).
\label{form-fourier}
\end{eqnarray}
It is well known that $\calE$ is a transient Dirichlet form and  is related (via Kato representation theorem) to
the selfadjoint
operator, commonly named the fractional Laplacian on  $\R^d$, and which we shall denote by  $(-\Delta)^{\alpha/2}$. We note that the domain of $(-\Delta)^{\alpha/2}$ is the fractional Sobolev space  $W^{\alpha,2}(\R^d)$. Having formula (\ref{form-fourier}) in hands one can explicitly evaluate the fractional Laplacian:
\begin{eqnarray}
(-\Delta)^{\alpha/2}f(x)=\mathcal{F}^{-1}\big(|\cdot|^{\alpha}\hat f\big)(x),\ \forall\,f\in W^{\alpha,2}(\R^d).
\label{operator1}
\end{eqnarray}
%
%
%
For an alternative and very interesting construction of the fractional Laplacian as a boundary operator we refer the reader to the paper \cite{Caf-Syl}.\\
%
%
%
Set  $L_0:=(-\Delta)^{\alpha/2}|_\Om$ the localization of  $(-\Delta)^{\alpha/2}$ on $\Om$, i.e., the operator which Dirichlet form in $L^2(\overline\Om,dx)$ is given by
\begin{eqnarray*}
D(\calE_\Om)&=&W_0^{\alpha/2,2}(\Om)\colon=\{f\in W^{\alpha/2,2}(\R^d)\colon\, f=0 ~~~q. e.~on~\Om^c\}\nonumber\\
\calE_\Om(f,g)&=&\calE(f,g)\nonumber\\
&=&\frac{1}{2}\A{(d,\alpha)}\int_\Om\int_\Om \frac{(f(x)-f(y))(g(x)-g(y))}{|x-y|^{d+\alpha}}\,dx\,dy
+\int_\Om f(x)g(x)\kappa_\Om(x)\,dx,
\end{eqnarray*}
where
\begin{eqnarray}
\kappa_\Om(x):=\A(d,\alpha)\int_{\Om^c}\frac{1}{|x-y|^{d+\alpha}}\,dy,
\end{eqnarray}
and the notation q.e. means quasi-everywhere with respect to the $\calE$-capacity.\\
The Dirichlet form $\calE_\Om$ is  regular, transient and the associated semigroup is irreducible even if $\Om$ is not connected.\\
Let $p_t(x,y),\ t>0,\ x,y\in\Om$ be the heat kernel of the semigroup $e^{-tL_0}$. It is known (see for instance \cite{bogdan-heatKernel}) that $p$ is jointly continuous on $(0,\infty)\times\Om\times\Om$ and
\begin{eqnarray}
0<p_t(x,y)=p_t(y,x)\leq\frac{C}{t^{d/\alpha}},\ \forall\,t>0,\ x,y\in\Om.
\label{FreeKern}
\end{eqnarray}
For every   $0<T\leq\infty$, we set $C_c^\infty\big([0,T)\times\Om\big)$  the usual space of infinitely differentiable functions on $[0,T)\times\Om$, having
compact support in $[0,T)\times\Om$.\\
We are in position at this stage to give  the notion of solution for the heat equation (\ref{heat1}).
\begin{defi}\label{sol-def}~~~~~~~~~~~~~~~~~~~~~~~~~~~~\\
{\rm Let $V\in L^1_{loc}(\Om)$ be nonnegative, $u_0\in L^2(\Om)$ be nonnegative as well and $0<T\leq\infty$.  We say that a Borel measurable function $u:[0,T)\times\R^d\to\R$ is a solution of the heat equation
\begin{eqnarray}
\label{heat2}
\left\{\begin{gathered}
-\frac{\partial u}{\partial t}=L_0 u - Vu,
\quad \hbox{in } (0,T)\times\Omega,\\
u(t,\cdot)=0,\ on~~~\Omega^c,\ \forall\,0<t<T\leq\infty\\
u(0,x)= u_{0}(x),~~~{\rm for}\ a.e.\ x\in \Omega,
\end{gathered}
\right.
\end{eqnarray}
if
\begin{enumerate}
\item $u\in\calL_{loc}^2\big([0,T), L_{loc}^2(\Om)\big)$, where $\calL^2$ is the Lebesgue space of square integrable functions.
\item $u\in L^{1}_{loc}\big((0,T)\times \Om,dt\otimes V\,dx\big)$.
\item For every $t>0$, $u(t,\cdot)= 0,\ a.e.$ on $\Om^c$.
\item For every $0\leq t< T$ and every Borel function $\phi:[0,T)\times\R^d$ such that $supp\,\phi\subset [0,T)\times\Om$, $\phi,\ \frac{\partial \phi}{\partial t}\in L^2((0,T)\times\Om)$, $\phi(t,\cdot)\in D(L_0),\ \forall\,t$ and
$$\int_0^t\int_\Om u(s,x)L_0\phi(s,x)\,ds\,dx<\infty$$
the following identity holds true
\begin{eqnarray}
\int_\Om \big((u\phi)(t,x)-u_0(x)\phi(0,x)\big)\,dx &+&\int_{0}^{t}\int_\Om
u(s,x)(-\phi_{s}(s,x)+L_0\phi(s,x))\,dx\,ds\nonumber\\
&=&\int_{0}^{t}\int_\Om u(s,x)\phi(s,x)V(x)\,dx\,ds.
\label{variational}
\end{eqnarray}
\end{enumerate}
}
\end{defi}
A function $\phi$ as indicated  in Definition \ref{sol-def} will be called a test function.\\
Henceforth our main task is to demonstrate that existence and blow-up of nonnegative solutions of the mentioned heat equation is deeply related to the size of
\begin{eqnarray}
 \displaystyle\lambda_{0}^{V}:=\inf_{\phi\in C_c^\infty(\Om)\setminus\{0\}}
 \frac{\calE_{\Omega}[\phi]-\int \phi^{2}\,V\,dx}{\int \phi^{2}\,dx}.
\end{eqnarray}
To that we shall need some preparing results.
\begin{prop}
Assume that $V\in L^\infty(\Om)$. Set $L_V$ the selfadjoint operator associated to the closed quadratic form
\begin{eqnarray}
\calE^V: D(\calE^V)=\Wo,\ \calE^V[f]=\calE_{\Om}[f] - \int_\Om f^2V\,dx,
\end{eqnarray}
in $L^2(\Om)$ and  $u(t):=e^{-tL_V}u_0,\ t\geq 0$. Then
\begin{enumerate}
\item $u(t)$ is a nonnegative global solution of problem (\ref{heat1}). Furthermore it satisfies Duhamel's formula
\begin{eqnarray}
u(t)= e^{-tL_0}u_0 + \int_0^t e^{-(t-s)L_0}(u(s)V)\,ds,\ \forall\,t\geq 0.
\label{Duha}
\end{eqnarray}
\item For every $t>0$, $\frac{\partial u}{\partial t}\in L^\infty(\Om)$ and
\begin{eqnarray}
\|\frac{\partial u}{\partial t} \|_{L^\infty(\Om)}\leq \frac{C}{t^{\frac{d}{2\alpha}}}\big(\frac{2}{t}+\|V\|_{L^{\infty}(\Om)}\big)\|u_0\|_{L^2(\Om)}.
\label{DerivEst}
\end{eqnarray}
\end{enumerate}
\label{existence1}
\end{prop}
\begin{proof}
Obviously $\calE_\Om$ is closed and hence we associate to it in a unique manner a selfadjoint operator which we denote by $L_V$. From the relationship between forms and semigroups we infer that $e^{-tL_V},\ t\geq 0$ is strongly continuous and we learn from \cite[Theorem 1.24, p.492]{kato} that $e^{-tL_V},\ t>0$ is holomorphic. Thereby   $u(t)=e^{-tL_V}u_0,\ t\geq 0$ enjoys the following properties
\begin{eqnarray}
u(t)\in D(L_V)=D(L_0),\ u\in C\big([0,\infty),L^2(\Om)\big)\cap C^1\big((0,\infty),L^2(\Om)\big),
\end{eqnarray}
and by \cite[Remark 1.21, p.492]{kato} $u(t)$ solves the equation (in the classical sense)
\begin{eqnarray}
\left\{\begin{gathered}
-\frac{\partial u}{\partial t}=L_0u - Vu,\
\quad \hbox{ }\ \forall\,t>0,\ a.e.\ x\in\Omega,\\
u(t,\cdot)=0,\ on~~~\Omega^c,\ \forall\,0<t<T\leq\infty\\
u(0,x)= u_{0}(x),~~~{\rm for}\ a.e.\ x\in \Omega,
\end{gathered}
\right.
\end{eqnarray}
Thus multiplying by a test function and integrating  we obtain that $u(t)$ is a solution of the heat equation in the sense of Definition \ref{sol-def}.\\
The fact that $u$ satisfies Duhamel formula is well known fact from the semigroup theory.\\
It remains to prove that $u$ is in fact nonnegative.\\
We observe that for every $f\in \Wo,\ |f|\in\Wo$ and since $\calE_\Om$ is a Dirichlet form it holds that $\calE^V[|f|]\leq \calE^V[f]$. Thus the semigroup $e^{-tL_V},\ t\geq 0$ is positivity preserving. Finally making use of Duhamel formula and recalling that $p_t>0\ on\ \Om\times\Om$ we conclude that $u(t)$ is nonnegative.\\
Let us prove the second assertion. Let $t>0$. Then
\begin{eqnarray}
\frac{\partial u}{\partial t}=-L_Ve^{-tL_V}u_0=-e^{-t/2L_V}L_Ve^{-t/2L_V}u_0.
\end{eqnarray}
Making use of Feynman-Kac formula we obtain that $e^{-tL_V},\ t>0$ maps continuously $L^2(\Om)$ into $L^\infty(\Om)$ together with the estimate
\begin{eqnarray}
\|e^{-t/2L_V}\|_{L^2,L^\infty}\leq \frac{C}{t^{d/{2\alpha}}}e^{t/4\|V\|_{L^\infty}}.
\end{eqnarray}
We thereby achieve
\begin{eqnarray}
\|\frac{\partial u}{\partial t}\|_{L^\infty(\Om)}\leq \|e^{-t/2L_V}\|_{L^2,L^\infty}\|L_Ve^{-t/2L_V}\|\|u_0\|_{L^2(\Om)}.
\end{eqnarray}
Making use of the spectral theorem we obtain $\|L_Ve^{-t/2L_V}\|\leq \frac{2}{t}+\|V\|_{L^{\infty}(\Om)}$. Finally putting all together we obtain the upper bound (\ref{DerivEst}).

\end{proof}
From now on, we set  $V_k:=V\wedge k$ and we denote by $(P_k)$ the heat equation corresponding to the Dirichlet fractional Laplacian perturbed by $-V_k$ instead of  $-V$:
\begin{eqnarray}
\label{heat-app}
(P_k)\colon\left\{\begin{gathered}
-\frac{\partial u}{\partial t}=L_0u - V_k u,
\quad \hbox{in } (0,T)\times\Omega,\\
u(t,\cdot)=0,\ on~~~\Omega^c,\ \forall\,0<t<T\leq\infty,\\
u(0,\cdot)= u_{0},~~~{\rm }\ \in \Omega,
\end{gathered}
\right.
\end{eqnarray}
Denote by  $L_k$  the selfadjoint operator associated to the closed quadratic form  $\calE_\Om-V_k$  and $u_k(t):=e^{-tL_k}u_0,\ t\geq 0$ the solution of problem $(P_k)$ given by Proposition \ref{existence1}. Then
$u_k$ satisfies Duhamel's formula:
\begin{eqnarray}
u_k(t,x)&=&e^{-tL_0}u_0(x)+\int_0^t\int_\Om p_{t-s}(x,y)u_k(s,x)V_k(y)\,dy\,ds,\ \forall\,t>0
\label{duhamel}
\end{eqnarray}
where $p_t,\ t>0$ is the heat kernel of the operator $e^{-tL_0}$.\\
The following  properties of the sequence $(u_k)$ are crucial for the later development of the paper.
\begin{lem}
\begin{itemize}
\item[i)] The sequence $(u_k)$ is increasing.
\item[ii)] If problem (\ref{heat1}) has a nonnegative  solution $u$ then $u_k\leq u,\ \forall\,k$. Moreover $\lim_{k\to\infty}u_k$ is a nonnegative solution of problem (\ref{heat1}) as well.
\end{itemize}
\label{domination}
\end{lem}
\begin{proof}
i) By Duhamel's formula, one has
\begin{eqnarray}
u_{k+1}(t)-u_k(t)&=&e^{-tL_{k+1}}u_0-e^{-tL_{k}}u_0=
\int_0^te^{-(t-s)L_{k}}e^{-sL_{k+1}}(u_0V_{k+1}-u_0V_k)(s)\,ds\nonumber\\
&\geq& 0.
\end{eqnarray}
ii) We follow an idea of Goldstein--Goldstein--Rhandi \cite[Prop.4.1]{rhandi}. Let $u$ be as stated in the lemma,  $0<t<T$ be fixed and  $\phi$ be positive test function  such that $Supp\,\phi\subset [0,t]\times\Om$.\\
From the definition of a solution  we infer
\begin{eqnarray}
\int_0^t \int_\Om (u_k(s)-u(s))(-\phi_s(s)+L_0\phi(s)-V_k\phi(s))\,ds\,dx&=&\int _0^t \int_\Om u\phi(V_k-V)\,ds\,dx\nonumber\\
&\leq& 0.
\label{negative1}
\end{eqnarray}
Let $\psi\in C_c^\infty\big((0,t)\times\Om\big) $ be nonnegative and consider the parabolic problem: find a positive test function $\phi$ solving the equation
\begin{eqnarray}
-\frac{\partial\phi}{\partial s} = -L_0\phi + V_k\phi +\psi\ {\rm in}\ (0,t)\times\Om,\ \phi(t,\cdot)=0.
\label{S1}
\end{eqnarray}
Then the latter problem has a positive solution which is given by (see \cite[Theorem 1.27, p.493]{kato})
\begin{eqnarray}
 \phi(s)= \int_0^{t-s} e^{-(t-s-\xi) (L_0-V_k)}\psi(t-\xi)\,d\xi,\ 0\leq s\leq t,\ \phi(s)=0,\ \forall\,s>t,
\end{eqnarray}
Plugging into equation (\ref{negative1}) yields
\begin{eqnarray}
\int_0^t \int_\Om (u_k-u)\psi\,ds\,dx\leq 0,\ \forall\, 0\leq\psi\in C_c^\infty\big((0,t)\times\Om\big).
\end{eqnarray}
As $t$ is arbitrary we obtain  $u_k\leq u$.\\
Let us prove that the limit of the sequence $(u_k)$ is a nonnegative solution.\\
Set $\ui:=\lim_{k\to\infty}u_k$. Obviously $\ui(t,\cdot)=0\ a.e.$ on $\Om^c$ for every $t\in(0,T)$. On the
other hand we have by the first step of  (ii), $0\leq\ui\leq u$ and therefore
$$
\ui\in \calL_{loc}^2\big((0,T),L_{loc}^2(\Omega)\big)\cap
L^1_{loc}\big([0,T)\times\Omega,dt\otimes V\,dx\big).
$$
Being solution of the heat equation $(P_k)$, the $u_k$'s satisfy: for every $0\leq t<T$, and every test function $\phi$,
\begin{eqnarray}
\int_\Om \big((u_k\phi)(t,x)-u_0(x)\phi(0,x)\big)\,dx +\int_{0}^{t}\int_\Om
u_k(s,x)\big(-\phi_{s}(s,x)+L_0\phi(s,x)\big)\,dx\,ds\nonumber\\
=\int_{0}^{t}\int_\Om u_k(s,x)\phi(s,x)V_k(x)\,dx\,ds.
\end{eqnarray}
By dominated convergence theorem we conclude that $\ui$ satisfies equation (\ref{variational}) as well, which ends the proof.
\end{proof}
The following lemma is inspired from the 'gradient' case where integration by parts is used. It extends without major difficulties for abstract nonlocal Dirichlet forms.
\begin{lem} Let $u_k$ be the nonnegative solution of the approximate problem $(P_k)$ and $\phi\in\Wa_0(\Om)\cap L^\infty$
which  support lies in $\Om$. Then $\frac{\phi^2}{u_{k}}\in\Wa_0(\Omega)$ and
\begin{eqnarray}
\calE_{\Omega}(u_{k},\frac{\phi^{2}}{u_{k}})\leq \calE_{\Omega}[\phi].
\label{en-in}
\end{eqnarray}
\label{enregy-comprison}
\end{lem}
\begin{proof}
It suffices to give the proof for positive $\phi$. Let $\phi\geq 0$ and $u_k$ be as specified in the lemma. As $\calE_\Om$ is a Dirichlet form,  to prove the first part it suffices to prove that $\frac{\phi}{u_k}\in\Wa_0(\Om)\cap L^\infty$.\\
Clearly for every compact subset $K\subset\Om$ there is a  constant $\kappa_k>0$ such that $u_k\geq\kappa_k$ on $K$ and then $\frac{\phi}{u_k}\in L^\infty$. To show that the latter function has finite energy we shall proceed directly.\\
An elementary computation yields
\begin{eqnarray}
\frac{\phi(x)}{u_{k}(x)}-\frac{\phi(y)}{u_{k}(y)}&=& \phi(x)\Big(\frac{1}{u_{k}(x)}-\frac{1}{u_{k}(y)}\Big) + \frac{1}{u_{k}(y)}\Big(\phi(x)-\phi(y)\Big)\nonumber\\
&=&\frac{\phi(x)}{u_{k}(x)u_{k}(y)}\Big(u_{k}(y)-u_{k}(x)\Big) + \frac{1}{u_{k}(y)}\Big(\phi(x)-\phi(y)\Big),
\end{eqnarray}
leading to
\begin{eqnarray}
\Big(\frac{\phi(x)}{u_{k}(x)}-\frac{\phi(y)}{u_{k}(y)}\Big)^{2}
&\leq&\frac{2\phi^{2}(x)}{u_{k}^{2}(x)u_{k}^{2}(y)}\Big(u_{k}(y)-u_{k}(x)\Big)^{2}
+\frac{2}{u_{k}^{2}(y)}\Big(\phi(x)-\phi(y)\Big)^{2}\nonumber\\
&\leq &2\max(\frac{\|\phi\|_\infty^{2}}{\kappa_{k}^{4}},\frac{1}{\kappa_{k}^{2}})
\Big[\Big(u_{k}(y)-u_{k}(x)\Big)^{2}
+\Big(\phi(x)-\phi(y)\Big)^{2}\Big]\nonumber \\
&\leq& C\Big[\Big(u_{k}(y)-u_{k}(x)\Big)^{2} +\Big(\phi(x)-\phi(y)\Big)^{2}\Big].
\end{eqnarray}
Finally we obtain
\begin{eqnarray}
\calE_\Omega[\frac{\phi}{u_k}] \leq C\big(\calE_\Om[u_k] +
\calE_\Omega[\phi]\big)<\infty,
\end{eqnarray}
yielding that  $\frac{\phi}{u_{k}}\in W_{0}^{\frac{\alpha}{2},2}(\Omega)$.\\
We proceed now to prove inequality (\ref{en-in}). A straightforward computation yields
\begin{eqnarray}
(u_k(x)-u_k(y))
\big(\frac{\phi^2(x)}{u_{k}(x)}-\frac{\phi^2(y)}{u_{k}(y)}\big)&=&\phi^2(x)+\phi^2(y)
\nonumber\\
&-&\frac{u_k(x)}{u_k(y)}\phi^2(y)-\frac{u_k(y)}{u_k(x)}\phi^2(x)\nonumber\\
&=&\phi^2(x)+\phi^2(y)-\frac{u_k^2(x)\phi^2(y)+u_k^2(y)\phi^2(x)}{u_k(x)u_k(y)}
\nonumber\\
&\leq&(\phi(x)-\phi(y))^2.
\end{eqnarray}
Thus
\begin{eqnarray}
\calE_\Om(u_k,\frac{\phi^2}{u_k})&=&\frac{1}{2}\mathcal{A}(d,\alpha)\int_\Omega
\int_\Omega\frac{(u_k(x)-u_k(y))
(\frac{\phi^2}{u_{k}}(x)-\frac{\phi^2}{u_{k}}(y))
}{|x-y|^{d+\alpha}}\,dx\,dy + \int_\Omega \phi^2(x)\kappa_\Omega(x)\,dx\nonumber\\
&\leq& \calE_\Omega[\phi],
\end{eqnarray}
which was to be proved.
\end{proof}
By the end of this section, we give a technical  result dealing  about the comparability of the ground state of the operator $L_0$ that will be needed in the proof of the nonexistence part.
\begin{lem}
Set $\varphi_0>0$ the normalized ground state of the operator $L_0$ and $h(t,x):=e^{-tL_0}u_0(x)$ for every $t>0$ and every $x\in\Omega$. Then
\begin{eqnarray}
h(t,\cdot)\sim\varphi_0,\ {\rm for\ every\ fixed}\ t>0.
\end{eqnarray}
\label{comparability}
\end{lem}
\begin{proof}
By a result due to Kulczycki \cite{kul} the operator $e^{-tL_0},\ t>0$ is intrinsically ultracontractive. Hence $p_t(x,y)\leq c_t\varphi_0(x)\varphi_0(y)$, which leads to $h(t,x)\leq c_t\varphi_0(x)$.\\
The reversed inequality follows from the intrinsic ultracontractivity as well (see \cite[p.345]{davies-simon}).
\end{proof}
\section{Existence of nonnegative solutions}
\begin{theo} Assume that $\lambda_0^V>-\infty$. Then the heat equation (\ref{heat1}) has at least one nonnegative global solution.
\label{suffi}
\end{theo}
The substance of Theorem \ref{suffi} may be established using \cite[Proposition 2.1]{stollmann} or
\cite[Proposition 5.7]{voigt}. However, for the convenience of the reader we shall give an adapted proof.
\begin{proof} We follow the local case \cite{cabre-martel,goldstein-zhang}.\\
Let $u_{n}(t)=e^{-tL_n}u_0,\ t\geq 0$. By Proposition \ref{existence1} and its proof, $u_n$ is a global solution of the approximate problem $(P_n)$ and satisfies
\begin{eqnarray}
\frac{d}{dt}\|u_{n}\|_{L^{2}(\Omega)}^{2}&=&-2(L_nu_n,u_n)\leq -2\lambda_{0}^{V}\int_{\Omega}u_{n}^{2}(t,x)\,dx.
\end{eqnarray}
The latter inequality is an immediate consequence of the finiteness  of $\lambda_0^V$. Hence by Gronwall's lemma
we achieve the upper estimate
 \begin{equation}
 \|u_{n}\|_{L^{2}(\Omega)}\leq \|u_{0}\|_{L^{2}(\Omega)}e^{-\lambda_{0}^{V}t}, \forall\,t\geq 0.
\label{uniform-bound}
\end{equation}
Thus the sequence  $(u_n)$  increases to a nonnegative  function $u$ for every $t\geq 0$ and a.e. $x\in\R^d$.
Furthermore  $u=0,\ a.e.$ on $\Omega^c$ and $u\in\calL_{loc}^2\big((0,\infty), L^2(\Omega)\big)$.\\
We are in position now to prove that $u$ solves the heat equation (\ref{heat1}). Indeed, having  Duhamel's formula
for the $u_n$'s in hands, we conclude by monotone convergence theorem  that
\begin{eqnarray}
u(t,x)&=& e^{-tL_0}u_0(x)+\int_0^t\int_{\Omega}p_{t-s}(x,y)u(s,y)\,V\,dy\,ds,
\end{eqnarray}
Since  $p_t>0,\ t>0$ on $\Om\times\Om$ and is jointly continuous the latter formula implies, that
$u\in L_{loc}^1\big((0,\infty)\times\Omega,dt\otimes V\,dx\big)$.\\
Now utilizing the equation fulfilled by the $u_n$'s being solutions of the $P_n$'s we obtain by dominated
convergence theorem, for every $t\geq 0$ and every test function $\phi$,
\begin{eqnarray}
\int_\Om u\phi|_{0}^{t}\,dx +\int_{0}^{t}\int_{\Om}
u(t,x)(-\phi_{t}(t,x)+L_0\phi(t,x))\,dx\,dt\nonumber\\
=\int_{0}^{t}\int_{\Om} u(t,x)\phi(t,x)V(x)\,dx\,dt.
\end{eqnarray}
\end{proof}
\begin{rk}{\rm
For every $u_0\in L^2(\Om)$ such that $u_0\geq 0$ and every $t\geq 0$, we define
\begin{eqnarray}
T_t u_0:= u(t,\cdot),
\end{eqnarray}
where $u(t,x)$ is the solution constructed in the latter proof. Then
\begin{eqnarray}
\vert|T_t u_0 \vert|_{L^2(\Om)}\leq e^{-\lam_0^V t}\vert|u_0\vert|_{L^2(\Om)},\ \forall\,t\geq 0.
\end{eqnarray}
Now for each $u_0\in L^2(\Om)$ and every $t\geq 0$, we set
\begin{eqnarray}
T_t u_0:= T_t u_0^+ - T_t u_0^-.
\end{eqnarray}
Then an elementary computation yields
\begin{eqnarray}
\vert|T_t u_0 \vert|_{L^2(\Om)}\leq e^{-\lam_0^V t}\vert|u_0\vert|_{L^2(\Om)},\ \forall\,t\geq 0.
\end{eqnarray}
Thereby, we  construct a family of linear  selfadjoint operators
\begin{eqnarray}
T_t\colon\, L^2(\Om)\to L^2(\Om),\ u_0\mapsto T_t u_0,
\end{eqnarray}
such that the operator norm of $T_t$ satisfies
\begin{eqnarray}
\|T_t\|\leq e^{-\lambda_0^V t},\ \forall\,t\geq 0.
\label{semigroup}
\end{eqnarray}
Furthermore the family  $(T_t)_{t\geq 0}$ satisfies the semigroup property and is strongly continuous.
}
\end{rk}

\section{Blow-up of nonnegative solutions}
%
In order to prove the blow-up part we are going first, to establish an estimate for the integral
$\int \ln u$, whenever $u$ is a nonnegative solution. Such estimate has an independent interest and is involved
to derive regularity properties for the solutions. Furthermore we shall use it to give necessary condition for the existence of exponentially bounded solutions.  Its use in our context is inspired from the one corresponding to the Dirichlet Laplacian (see \cite{cabre-martel,goldstein-zhang}).
\begin{theo}
Assume that $u$ is a nonnegative solution of the heat equation (\ref{heat1}). Then
for all $0<t_{1}<t_{2}<T$, and all $\Phi\in C_{c}^{\infty}(\Omega)$, we have
\begin{equation}\label{log-sol}
\int_{\Omega} \Phi^{2}\,V\,dx- \calE_{\Omega}[\Phi]
\leq \frac{1}{t_{2}-t_{1}}\int_{\Omega}\ln\Big(\frac{{u}(t_{2})}{{u}(t_{1})}\Big)
\Phi^{2}\,dx.
\end{equation}
\label{Log-Sol}
\end{theo}
\begin{proof}
Pick a function $\Phi\in C_c^\infty(\Omega)$. Without loss of generality we may and shall suppose that $\int \Phi^2\,dx=1$.\\
Let $u_n(t)$ be the nonnegative solution of problem $(P_n)$ given by  $u_n(t)=e^{-tL_n}u_0,\ t\geq 0$. We already know that
\begin{eqnarray}
-\frac{\partial}{\partial t} u_n(t)=L_nu_n(t),\ t>0.
\label{derivative}
\end{eqnarray}
On the other hand by Lemma \ref{enregy-comprison} we have $\frac{\Phi^{2}}{u_{n}}\in W_0^{\alpha/2,2}(\Omega)$. Thus multiplying the latter identity by $\frac{\Phi^{2}}{u_{n}}$ and integrating over $\Om$ yields
\begin{eqnarray}
\int_\Om \Phi^{2}\,V_n\,dx&=&\int_\Om \frac{\partial u_n}{\partial t}\frac{\Phi^{2}}{u_{n}}dx
+ \calE_\Omega(u_{n},\frac{\Phi^{2}}{u_{n}}).
\label{inter}
\end{eqnarray}
Let us recall that from Duhamel formula we derive $u_n(t)\geq \int_\Om p_t(x,y)u_0(y)\,dy$ for each $t>0$. Thus from the continuity  of $p_t(x,y),\ t>0$ in $t,x$ and $y$ together with the fact that $p_t>0$ we conclude: for every $0<a\leq b<\infty$ and every compact subset $K\subset\Om$,
\begin{eqnarray}
\inf_{a\leq t\leq b,\ x\in K}u_n(t,x)>C_n(K,a,b)>0.
\end{eqnarray}
Moreover from  estimate (\ref{DerivEst}) we infer
\begin{eqnarray}
\sup_{a\leq t\leq b}\|\frac{\partial u_n}{\partial t}\|_{L^\infty(\Om)}\leq C_n(a,b)<\infty.
\end{eqnarray}
The use of these both facts enables us to interchange derivation and integration in(\ref{inter}) to obtain with the help of  the energy estimate from Lemma \ref{enregy-comprison}
\begin{eqnarray}
\int_\Om \Phi^{2}\,V_n\,dx &=&\frac{d}{dt}\int_\Om(\ln u_{n})\Phi^{2} dx + \calE_\Omega(u_{n},\frac{\Phi^{2}}{u_{n}})\nonumber\\
&\leq& \frac{d}{dt}\int_\Om(\ln u_{n})\Phi^{2}\,dx + \calE_\Omega[\Phi].
\end{eqnarray}
Hence, integrating between $t_{1}$ and $t_{2}$ we achieve,
\begin{eqnarray}
\int_{\Omega} \Phi^{2}\,V_n\,dx - \calE_\Omega[\Phi]\leq \frac{1}{t_2 - t_1}\int_{\Omega}\ln\Big(\frac{{u_n}(t_{2})}{{u_n}(t_{1})}\Big)
\Phi^{2}\,dx.
\end{eqnarray}
On the other hand using Lemma \ref{domination} together with Jensen's inequality we achieve
\begin{eqnarray}
-\infty<\int_\Om \ln u_n(t_i)\,\Phi^2dx\leq \ln\big(\int_\Om  u_n(t_i)\,\Phi^2dx\big)\leq
\ln\big(\int_\Om u(t_i)\,\Phi^2dx\big)<\infty,\ i=1,2.
\end{eqnarray}
Finally, we  pass to the limit and use monotone convergence theorem to obtain inequality (\ref{log-sol}), which finishes the proof.
\end{proof}
On the light of Theorem \ref{suffi} together with Theorem \ref{Log-Sol}, we get a necessary and sufficient condition for the existence of a nonnegative exponentially bounded global solution, i.e. a nonnegative global solution such that there is $C>0,\ \omega\in\R$ with
\begin{eqnarray}
\|u(t)\|_{L^2(\Om)}\leq Ce^{\omega t},\ \forall\,t>0.
\end{eqnarray}
Let us mention that the following proposition was also proved in \cite[Theorem 4.2]{keller-lenz}, in a more general context, however with a different proof.
\begin{prop}
The heat equation (\ref{heat1}) has a nonnegative exponentially bounded global solution if and only if $\lam_0^V>-\infty$.
\end{prop}
\begin{proof}
We have already established the sufficiency part in Theorem \ref{suffi}.\\
Conversely, assume that equation (\ref{heat1}) has a nonnegative solution $u$ such that
$$
\|u(t)\|_{L^2(\Om)}\leq Ce^{\omega t},\ \forall\,t>0.
$$
Choosing in Theorem \ref{Log-Sol} $t_2=t>t_1=1$, $\Phi\in C_c^\infty(\Om)$ with $\int \Phi^2\,dx=1$ and apply Jensen's inequality leads to
\begin{eqnarray}
\int_{\Omega} \Phi^{2}\,V\,dx- \calE_{\Omega}[\Phi]
 & \leq &\frac{1}{t-1}\ln\big(\int_\Omega u(t)\Phi^2\,dx\big)
-\frac{1}{t-1} \int \ln(u(1))\Phi^2\,dx\nonumber\\
&\leq & \frac{\ln(Ce^{wt})}{t-1} +\frac{\ln(\|\Phi^4\|_{L^2(\Om)})}{t-1} - \frac{1}{t-1} \int \ln(u(1))\Phi^2\,dx.
\end{eqnarray}
Letting $t\to\infty$, yields $\lam_0^V>-\infty$.
\end{proof}
We are in position yet, to give a condition ensuring absence of nonnegative solution as well as instantaneous blow up.
\begin{theo} Assume that $\lambda_0^{(1-\epsilon)V}=-\infty$ for some $\epsilon>0$. Then the heat equation (\ref{heat1}) has no nonnegative solution. Moreover all nonnegative solutions  blow up completely and instantaneously, i.e.:
$\lim_{n\to\infty}u_n(t,x)=\infty$ for every $t>0$ and every $x\in\Omega$, where $u_n$ is the solution of $(P_n)$.
\label{necessary}
\end{theo}
\begin{proof}
Assume that a nonnegative solution $u$  exists. Relying
on Lemma \ref{domination} we may
and shall assume that $u$ is the increasing limit of the $u_n$'s.\\
Let $0<\rho \in C_0(\Omega)$ (the space of continuous functions on $\Om$ vanishing on $\partial\Om$) be such that $\ln \rho \in L^{p}(\Omega)$ for any $p>1$.\\
{\em Claim}: There exists at most one point $t_{1}\in (0,T)$ such that $u(t_{1},\cdot)\rho\in L^{1}(\Omega)$. Indeed,
suppose that the  contrary holds true. Then there exist $t_{1},~t_{2}\in (0,T)$ such that $t_{2}>t_{1}$ and for $i=1,2$ $u(t_{i},\cdot)\rho\in L^{1}(\Omega)$.\\
For a small $\eta>0$, set
\begin{eqnarray}
\Omega_{\eta}\colon=\{x\in \Omega\colon\,\delta(x):=dist(x,\Om^c)\geq\eta\}.
\end{eqnarray}
Note that $u(t_{i},x)\rho(x)\geq c>0$ when $x\in \Omega_{\eta}$ for some $c$ depending on $t_{i}$ and $\eta$.
Thus using Jensen's inequality twice we first get that $\ln(u(t_i,\cdot)\rho)\in L^1(\Omega_\eta)$ and then
\begin{equation}\label{int1}
\begin{gathered}
\int_{\Omega_{\eta}}\big\vert \ln\big(u(t_{i},x)\rho(x)\big)\big\vert ^{p}dx<\infty.
\end{gathered}
\end{equation}
Now we decompose the set $\Omega\setminus\Omega_\eta$ into:
\begin{eqnarray}
\Omega\setminus\Omega_{\eta}&=&S_{1}\cup S_{2}\\
&=&\{x\in \Omega_{\eta}^{c}\colon\,u(t_{i},x)\rho(x)\geq m\}
\cup \{x\in \Omega_{\eta}^{c}\colon\,u(t_{i},x)\rho(x)< m\}.
\end{eqnarray}
Observing that $\ln^{p} s$ is a concave function of $s$, for $s\geq e^{p-1}$, we choose  $m$ sufficiently large and apply Jensen's inequality on $S_{1}$, to derive that, for any $p>1$,
\begin{eqnarray}\label{int2}
\int_{S_{1}}|\ln\Big(u(t_{i},x)\rho(x)\Big)|^{p}\,dx&\leq&
C\ln^{p}\Big(\int_{S_{1}}u(t_{i},x)\rho(x)\,dx\Big)<\infty.
\end{eqnarray}
For $x\in S_{2}$, we have
\begin{eqnarray}
m>u(t_{i},x)\rho(x)&=&\Big[\int_{\Omega}p_{t_{i}}(x,y)u_{0}(y)dy\\
&+&\int_{0}^{t_{i}}\int_{\Omega}p_{t_i-s}(x,y)u(y,s)\,V\,dy\,ds\Big]\rho(x)\nonumber\\
&\geq&\int_{\Omega}p_{t_{i}}(x,y)u_{0}(y)dy \rho(x)\\
&\colon=&h(t_{i},x)\rho(x),\ i=1,2.
\end{eqnarray}
Thus
\begin{eqnarray}
\ln m\geq \ln\big(u(t_{i},x)\rho(x)\big)&=&\ln\big(\frac{u(t_{i},x)\rho(x)}{h(t_{i},x)
\rho(x)}\big)
+\ln\big(h(t_{i},x)\rho(x)\big)\nonumber\\
&\geq& \ln\big(h(t_{i},x)\rho(x)\big),
\end{eqnarray}
leading to the estimate
\begin{eqnarray}
|\ln(u(t_{i},x)\rho(x))|\leq \vert\ln m\vert + |\ln(h(t_{i},x)\rho(x))|,\ i=1,2.
\end{eqnarray}
As $\Om$ is a Lipschitz domain,  it is known  that (see \cite[p. 78]{bogdan-book}) there are constants $C>0,\gamma>0$ depending solely on $\Om$ and $\alpha$ such that
\begin{eqnarray}
\varphi_0\geq C\delta^{\gamma}.
\label{ground-distance}
\end{eqnarray}
Thus by Lemma \ref{comparability} together with the lower bound (\ref{ground-distance}), the $p$-integrability of the function $\ln[h(t_{i},x)\rho(x)]$ reduces to the $p$-integrability of $\ln\delta(x)$.\\
Let us recall the known fact (see \cite[Theorem 3.3]{delfour}) that the distance function $\delta$ is $a.e.$ differentiable and that $|\nabla\delta|=1,\ a.e.$ on $\Om$. Making use of the generalized coarea formula (see \cite[Theorem 1.2.6, p.5]{edmunds}), we obtain
\begin{eqnarray}
\int_{S_2}|\ln\delta(x)|^p|\nabla\delta(x)|\,dx &=& \int_{S_2}|\ln\delta(x)|^p\,dx=\int_0^R\int_{S_2\cap\{\delta=t\}}|\ln t|^p\,dH^{d-1}(x)\,dt\nonumber\\
&\leq& C\int_{0}^R |\ln t|^p\,dt<\infty,\ {\rm for\ any}\ p>1,
\end{eqnarray}
where $R$ and $C$ are  finite constants.
%
%
Therefore $\ln(u(t_i,\cdot)\rho)\in L^p(\Omega),\ i=1,2$.\\
Now we conclude that
\begin{eqnarray}
\ln\frac{u(t_{2},\cdot)}{u(t_{1},\cdot)}=\ln\Big(u(t_{2},\cdot)\rho(\cdot)\Big)
-\ln\Big(u(t_{1},\cdot)\rho(\cdot)\Big)\in L^{p}(\Omega),\ \forall\,p>1.
\end{eqnarray}
On the other hand it is well known that, being in the space $L^{p}(\Omega)$ for $p>d/{\alpha}$, the
function $\ln\frac{u(t_{2},\cdot)}{u(t_{1},\cdot)}$ is in fact in the Kato class and whence it satisfies the following: For
any $r>0$, there exists $C(r)>0$ such that
 \begin{equation}\label{kato-estimate}
\begin{gathered}
\frac{1}{t_{2}-t_{1}}\int_{\Omega}\ln\frac{u(t_{2},x)}{u(t_{1},x)}\Phi^{2}(x)\,dx\leq r\calE_{\Omega}[\Phi]
+ C(r)\int_{\Omega}\Phi^{2}(x)\,dx,\ \forall\,\Phi\in C_c^{\infty}(\Omega).
\end{gathered}
\end{equation}
Having inequality (\ref{log-sol}) in hands, we achieve
\begin{eqnarray}
\int_{\Omega} \Phi^{2}(x)V\,dx - {\mathcal E}_{\Omega}[\Phi]
\leq r{\mathcal E}_{\Omega}[\Phi] + C(r)\int_{\Omega}\Phi^{2}(x)\,dx,\ \forall\,\Phi\in C_c^\infty(\Om).
\end{eqnarray}
Therefore for every $\Phi\in C_c^\infty(\Om)$ such that $\int_\Omega\Phi^2\,dx=1$, we have
\begin{eqnarray}
\frac{-C(r)}{1+r}\leq {{\mathcal E}_{\Omega}[\Phi]-(1+ r)^{-1}\int_{\Omega} \Phi^{2}\,V\,dx}.
\end{eqnarray}
Whence
\begin{eqnarray}
\lambda_{0}^{(1+r)^{-1}V}> -\infty,\  \forall\,r>0,
\end{eqnarray}
which  contradicts the assumption of the theorem and the claim is finally proved.\\
Given $x\in \Omega$ and $t\in (0,T)$, we take $\rho=\rho(y)=p_{\frac{t}{2}}(x,y)$. Owing to the sharp estimate
of Lemma \ref{comparability} together with the lower bound (\ref{ground-distance}), we conclude
that $\ln \rho \in L^{p}(\Omega)$ as was the case for $h$.\\
On the other hand from the properties of the heat kernel  for the Dirichlet fractional Laplacian, we
have $\rho(y)>0$ for any $y\in\Omega$.\\
If there is no $s\in (0,\frac{t}{2}]$ such that $\rho(\cdot)u(s,\cdot)\in L^{1}(\Omega)$, then by using Duhamel's principle, we have
\begin{eqnarray}
u(t,x)&=& e^{-\frac{t}{2}L_0}u(\frac{t}{2},x) + \int_{\frac{t}{2}}^{t}\int_{\Omega}p_{s}(x,y)u(s,y)V(y)\,dy ds\nonumber\\
&\geq& e^{-\frac{t}{2}L_0}u(\frac{t}{2},x)\nonumber\\
&=&\int_{\Omega}p_{\frac{t}{2}}(x,y)u(\frac{t}{2},y)\,dy
=\int_{\Omega}\rho(y)u(\frac{t}{2},y)\,dy=\infty.
\end{eqnarray}
In case $s\in(0,\frac{t}{2}]$ is the only point such that $\rho(\cdot)u(s,\cdot)\in L^{1}(\Omega)$, making use
of Duhamel's principle once again, we obtain
\begin{eqnarray}
u\big(\frac{(t+s)}{2},x\big)\geq\int_{\Omega}p_{\frac{t}{2}}(x,y)u(\frac{s}{2},y)\,dy
=\int_{\Omega}\rho(y)u(\frac{s}{2},y)\,dy
=\infty.
\end{eqnarray}
From the intrinsic ultracontractivity property for the semigroup $e^{-tL_0}$, we derive that for
every $x\in\Omega$, every $r>0$ such that $B_r(x)\subset\Omega$, every $z\in B_r(x)$, and every
small $\gamma>0$, there is a constant $c=c(t,\gamma,r)$ such that
\begin{eqnarray}
p_{\frac{t+\gamma }{2}}(z,y)\geq c p_{\frac{t}{2}}(x,y),\ \forall\,y\in \Omega.
\end{eqnarray}
Indeed
\begin{eqnarray}
p_{\frac{t+\gamma }{2}}(z,y)&\sim& \varphi_0(z)\varphi_0(y)\geq \frac{\inf_{B_r(x)}\varphi_0}{\sup_{B_r(x)}\varphi_0}
\varphi_0(x)\varphi_0(y)\nonumber\\
&\sim&p_{\frac{t}{2}}(x,y).
\end{eqnarray}
Making use of the latter claim we achieve, for every $z\in B_r(x)$:
\begin{eqnarray}
u(\frac{t+s+\gamma }{2},z)&\geq&e^{-(\frac{t+\gamma }{2})L_{0}}u(\frac{s}{2},z)
=\int_{\Omega}p_{\frac{t+\gamma }{2}}(z,y)u(\frac{s}{2},y)\,dy\nonumber\\
&\geq& c\int_{\Omega}p_{\frac{t}{2}}(x,y)u(\frac{s}{2},y)\,dy
\geq c\int_{\Omega}\rho(y)u(\frac{s}{2},y)\,dy
=\infty.
\end{eqnarray}
By the semigroup property (or Duhamel's formula) once again, we obtain
\begin{eqnarray}
u(t,x)&\geq&\int_{\Omega}p_{\frac{t-s-\gamma }{2}}(x,z)u(\frac{t+ s+\gamma }{2},z)dz\nonumber\\
&\geq&\int_{B_r(x)}p_{\frac{t-s-\gamma }{2}}(x,z)u(\frac{t+ s+\gamma }{2},z)\,dz = \infty,
\end{eqnarray}
as $u(\frac{t+ s+\gamma }{2},z)=\infty$ on $B_r(x)$. Since $(t,x)$ is arbitrary, this proves the blow-up and the proof is finished.
\end{proof}
\section{Examples}
In this section we provide some examples that support the already developed theory. %
\begin{lem} Assume that there is $\lambda>1$ and a sequence of balls $B_k\subset\Omega
$ such that their Lebesgue volumes  $|B_k|\downarrow 0$ and a sequence $(\phi_k)\subset C_c^\infty(\Om)$ with
$Supp\, \phi_k\subset B_k,\ \int \phi_k^2(x)\,dx=1,\ \forall\,k$ such that
\begin{eqnarray}
\int \phi^2_k(x)\,V\,dx\geq\lambda\calE_\Omega[\phi_k],\ \forall\,k.
\end{eqnarray}
Then the heat equation (\ref{heat1}), has no nonnegative solution.
\label{condition}
\end{lem}

\begin{proof} As a consequence of the condition given in the lemma, there is $\lambda'>1$ and $\epsilon\in (0,1)$ such that

\begin{eqnarray}
(1-\epsilon)\int \phi^2_k(x)\,V\,dx\geq\lambda'\calE_\Omega[\phi_k],\ \forall\,k.
\end{eqnarray}
Thus
\begin{eqnarray}
-\calE_{B_k}[\phi_k] + (1-\epsilon)\int \phi_k^2(x)\,V\,dx\geq(\lambda'-1)\calE_\Omega[\phi_k]\geq c|B_k|^{-\alpha/d}
,\ \forall\,k.
\end{eqnarray}
Hence $\lambda_0^{(1-\epsilon)V}(B_k)\leq -c|B_k|^{-\alpha/d}\to -\infty$, as $k\to\infty$. Now observing that
\begin{eqnarray}
\lambda_0^{(1-\epsilon)V}(B_k)\geq \lambda_0^{(1-\epsilon)V},
\end{eqnarray}
yields the result.
\end{proof}
\begin{exa}{\rm {\em Hardy potential with interior singularity}.\\
Let $\alpha<\min(2,d)$, and $\Om\subset\R^d$ an open bounded subset with Lipschitz boundary and containing $0$.
Set $V_c(x)=\frac{c}{|x|^\alpha},\ x\neq 0$ and $c\geq 0$. If
$0\leq c\leq c^*\colon=\frac{2^\alpha\Gamma^2(\frac{d+\alpha}{4})}
{ \Gamma^2(\frac{d-\alpha}{4})}$, then the heat equation associated to $L_{V_c}$ has  a nonnegative solution
(owing to Hardy's inequality).\\
However if $c>c^*$ then the heat equation has no nonnegative  solution. Indeed, owing to the sharpness of
the Hardy's inequality
\begin{eqnarray}
\int_\Om \frac{f^2(x)}{|x|^\alpha}\,dx\leq \frac{1}{c^*}\calE_\Om[f],\ \forall\,f\in W_0^{\alpha/2,2}(\Omega),
\label{hardy}
\end{eqnarray}
there is $\lambda>1$ and a function $\phi\in C_c^\infty(\Om)$  such that
\begin{eqnarray}
\int \phi^2(x)V(x)\,dx\geq\lambda\calE_\Omega[\phi].
\end{eqnarray}
By a scaling $x=\kappa x'$, we may assume that $supp\ \phi\subset B_R(0)$. Now an elementary computation shows that the sequence $\psi_k$ defined by $\psi_k(x)=\phi(kx)$ and $B_k:=B_{R/k}$
 fulfills the conditions of Lemma \ref{condition}.\\
Let us emphasize at this stage,  that while for $0\leq c<c^*$, $u(t)\in W_0^{\alpha/2,2}(\Om),\ \forall\,t>0$, it is not
the case for
the critical constant $c^*$. In fact, according to \cite[Theorem 4.2]{benamor-JPA} in the critical case we have
\begin{eqnarray}
 u(t)\sim |\cdot|^{-\frac{d-\alpha}{2}}\delta^{\alpha/2},\ {\rm for\ large}\ t,\ {\rm where}\ \delta(x)=dist(x,\Omega^c).
\end{eqnarray}
Using Hardy's inequality once again we recognize that $|\cdot|^{-\frac{d-\alpha}{2}}\delta^{\alpha/2}\notin
W_0^{\alpha/2,2}(\Om)$.
}
\label{HP-local}
\end{exa}
\begin{exa}{\rm {\em Hardy potential with boundary singularity}.\\
Assume that the following Hardy's inequality holds true
\begin{eqnarray}
\int \frac{f^2(x)}{\delta^{\alpha}(x)}\,dx\leq\frac{1}{\kappa}\calE_\Om[f],\ \forall\,f\in W_0^{\alpha/2,2}(\Omega),
\label{H-distance}
\end{eqnarray}
with sharp constant $1/\kappa^*$. Take $V_\kappa=\frac{\kappa}{\delta^{\alpha}(x)},\ \kappa\geq 0$. Arguing as in Example \ref{HP-local}, we conclude that for $\kappa>\kappa^*$, the related heat equation has no nonnegative solution, whereas it has for $\kappa\leq\kappa^*$.\\
According to \cite[Corollary 2.4]{chen-song}, inequality (\ref{H-distance}) is satisfied if $d\geq 2$ and $\alp\neq 1$.\\
Let us finally quote that the connection between Kato inequality and existence as well as nonexistence of positive solutions for the Laplacian on the half-space  with boundary singularity was discussed by Ishige--Ishiwata in \cite{ishige}. Similar results to our's were discovered in that paper.

}
\end{exa}

\bibliography{biblio-HeatFDL}

\end{document}